\documentclass[11pt]{amsart}
\usepackage{enumerate}
\usepackage{amssymb}
\usepackage{graphicx}
\usepackage{amscd, color}
\usepackage{amsmath}
\usepackage{amsfonts}
\usepackage[english]{babel}
\usepackage{epsfig}
\usepackage{xypic}
\usepackage{color}

\newtheorem{theorem}{Theorem}[section]
\newtheorem{proposition}[theorem]{Proposition}
\newtheorem{corollary}[theorem]{Corollary}
\newtheorem{example}[theorem]{Example}

\newtheorem{remark}[theorem]{Remark}

\newtheorem{lemma}[theorem]{Lemma}
\newtheorem{question}[theorem]{Question}

\DeclareMathOperator{\conv}{conv}

\begin{document}

\title[Local compactness in right bounded asymmetric  spaces]{Local compactness in right bounded asymmetric normed spaces}

\author{N. Jonard-P\'erez and  E.A. S\'anchez-P\'erez}

\subjclass[2010]{46A50, 46A55, 46B50, 52A07, 52A10}

\keywords{Asymmetric norm, Right bounded, Compactness,  Local compactness, Convex set}

\thanks{The first author has been supported by Conacyt grant 252849 (M\'exico) and by  PAPIIT grant IA104816  (UNAM, M\'exico).
The second author has been supported by Ministerio de  Econom\'{\i}a y Competitividad (Spain) (project MTM2016-77054-C2-1-P)}

\maketitle

\begin{abstract}
 We characterize  the finite dimensional asymmetric normed spaces which are right bounded and the relation of this property with the natural  compactness properties  of the unit ball, as compactness and strong compactness. In contrast with some results found in the existing literature, we show that not all right bounded asymmetric norms have compact closed balls. We also prove that there are finite dimensional asymmetric normed spaces that satisfy that the closed unit ball is compact, but not strongly compact, closing in this way an open question on the topology of finite dimensional asymmetric normed spaces.  In the positive direction, we will prove that a finite dimensional asymmetric normed space is strongly locally compact if and only if it is right bounded.
\end{abstract}

\section{Introduction}

It is well known that a normed vector space is locally compact if and only if it is finite dimensional.
However, in the asymmetric case this is no longer true.
In the context of asymmetric normed spaces, this matter is related to the relevant notion of right boundedness that has been widely used (see e.g. \cite{todos, Conradie 2}; see Section \S \ref{Sec3} for the definition). Right bounded asymmetric normed spaces were introduced in \cite[Definition 16]{luis}. In that same paper it was stated that the unitary closed ball of a right bounded asymmetric normed space is always compact (\cite[Proposition 17]{luis}). However,  we will show in Example~\ref{e:main 1} that ---except if the constant $r$ in the definition of right boundedness is $1$--- this result is not true.
Therefore,  it is natural to ask if it is possible to give a characterization of right boundedness for finite dimensional asymmetric normed spaces in terms of a weaker compactness-type property: local compactness.  Recall that a topological space $X$ is locally compact iff every point $x\in X$ has a local base consisting of compact neighborhoods.

The aim of this work is to solve that problem. Another relevant compactness property that can also be found in the literature ---strong compactness--, and its local version will be considered. As part of our main result, we will prove that  right boundedness and strong local compactness are equivalent notions in the class of asymmetric normed spaces. However, we will also show that there are finite dimensional asymmetric normed spaces of dimension $3$
such that the unitary closed ball is compact and right bounded, but not strongly compact (see Example \ref{inf}). 
 We will also prove other equivalences related with the geometry and the topology of the asymmetric unitary closed balls.

\section{Preliminaries}

Consider a real linear space $X$ and let $\mathbb R^+$ be the set of non-negative real numbers. An \textit{asymmetric norm} $q$ on $X$ is a function  $q:X \rightarrow \mathbb R^+$ such that
\begin{enumerate}[\rm(1)]
\item $q(ax)=aq(x)$ for every $x \in X$ and $a \in \mathbb R^+$,
\item $q(x+y) \le q(x) + q(y)$, $x,y \in X$, and
\item for every $x \in X$, if $q(x)=q(-x)=0$, then $x=0$.
\end{enumerate}

The pair $(X,q)$ is called an \textit{asymmetric normed linear space}.
 An asymmetric norm defines a non-symmetric  topology on $X$ that is given by the open balls
$B_\varepsilon^q(x):=\{y \in X: \, q(y-x) < \epsilon\}$. This
topology is denoted by $\tau_q$.

We will denote by $\theta _q$ the set of all $x\in X$ such that $q(x)=0$. The set $\theta_q$ is a convex cone; this means that $\mu x \in \theta_q$ and $x+y\in\theta_q$ for every $x,y\in\theta_q$ and $\mu\geq 0$. This set
 plays a fundamental role in many topological, geometric and analytic  results about asymmetric normed linear spaces. In particular, $(X,q)$ is $T_1$ if and only if $\theta_q=\{0\}.$

The following are well known results concerning the set $\theta_q$ (see \cite{luis} for the proofs) that will be needed in the paper.

\begin{itemize}
\item[(O1)] For any open set $U\subset X$, we always have that $U=U+\theta_q$.
\item [(O2)] $K\subset X$ is compact if and only if $K+\theta_q$ is compact.
\end{itemize}

For every asymmetric normed space $(X, q)$
 the map $q^{s}:X\to \mathbb R^{+}$ defined by the rule 
$$
q^s(x):= \max \{q(x),q(-x)\}, \quad x \in X,
$$
is a norm that generates a topology stronger than the one generated by $q$. We will use the symbols $B_\varepsilon^q$ and $B_\varepsilon^{q^s}$ to distinguish the balls of $(X,q)$ and $(X,q^s)$, respectively. More precisely, for every $x\in X$ and $\varepsilon >0$ we will use the following notations
\begin{align*}
&B_\varepsilon^q(x)=\{y\in X\mid q(y-x)<\varepsilon\},\\
&B_\varepsilon^q[x]=\{y\in X\mid q(y-x)\leq\varepsilon\},\\
&B_\varepsilon^{q^s}(x)=\{y\in X\mid q^s(y-x)<\varepsilon\},\\
&B_\varepsilon^{q^s}[x]=\{y\in X\mid q^s(y-x)\leq\varepsilon\}.\\
\end{align*}

The set $B_\varepsilon^q[x]$ is called the \textit{closed ball} of radius $\varepsilon$ centered at $x$. However, in general this set is not closed with respect to $\tau_q$.

 In order to avoid confusion, when necessary, we will say that a set is \textit{$q$-compact} ($q^{s}$-compact) if it is compact in the topology generated by $q$ ($q^{s}$).  We will use the expressions \textit{$q$-open} and \textit{$q$-closed} sets (\textit{$q^{s}$-open} and \textit{$q^{s}$-closed} sets)  in the same way.

A subset $K\subset X$ in an asymmetric normed space $(X,q)$ is \textit{strongly $q$-compact} (or simply, \textit{strongly compact}) if and only if there exists a $q^s$-compact set $S\subset X$  such that
$$
S\subset K\subset S+\theta_q.
$$
Every strongly compact set is $q$-compact, but the converse is not true (see e.g. \cite{todos, new, Jonard Sanchez 2}).
If each point of the asymmetric normed space $X$ has a local base consisting of strongly compact sets we will say that $X$ is \textit{strongly locally compact}. Evidently, if $X$ is strongly locally compact then it is also locally compact.

Since $q(x)\leq q^s(x)$ for every $x\in X$, we always have the following contentions
\begin{equation}\label{e:theta balls}
B_\varepsilon^{q^s}(x)+\theta_q\subset B_\varepsilon^q(x)\quad\text{ and }\quad
B_\varepsilon^{q^s}[x]+\theta_q\subset B_\varepsilon^q[x].
\end{equation}

Addition  in asymmetric normed spaces is always continuous but the scalar multiplication is not (see e.g. \cite{cobzas}). However it is well-known that multiplication by non negative scalars is continuous. For the convenience of the reader, we include the proof of this result that will be used often in the paper. 

\begin{lemma}\label{l:scalar multiplication continuous}
Let $(X,q)$ be an asymmetric normed space. The map $\mathbb R^+\times X$ given by $( \mu, x)\to \mu x$ is continuous if we endow $\mathbb R^+$ with the Euclidean topology.
\end{lemma}

\begin{proof}
Let $\varepsilon>0$ and $(\mu,x)\in \mathbb R^+\times X $. Then there exists $\delta_1>0$ such that $\delta_1 q^{s}(x)< \varepsilon/2$.
Since $\mu\geq 0$, if we define $\delta_2:=\varepsilon /2(\mu+\delta_1)$ we have that $\delta_2>0$.
Now take $\lambda \geq 0$ such that $|\lambda-\mu|<\delta_1$ and $y\in B_{\delta_2}^q(x)$. Thus
\begin{align*}
q(\lambda y-\mu x)&=q(\lambda y-\lambda x+\lambda x-\mu x)\\
&\leq q(\lambda y-\lambda x)+q(\lambda x-\mu x)\\
&\leq\lambda q(y-x)+q^s(\lambda x-\mu x)\\
&<(\mu+\delta_1)\frac{\varepsilon}{2(\mu+\delta_1)}+|\lambda-\mu |q^s(x)\\
&<\frac{\varepsilon}{2}+\delta_1 q^s(x)<\frac{\varepsilon}{2} + \frac{\varepsilon}{2} =\varepsilon.
\end{align*}
This proves the lemma.
\end{proof}

\subsection{Equivalent and right bounded asymmetric norms} \label{Sec3}

Let us show some basic facts on the relation among equivalent asymmetric norms and the right boundedness property for the corresponding spaces that will be used later on. Some of the proofs are straightforward, so  we will  write only some hints for getting them.

Let $X$ be a real vector space.  Two asymmetric norms in $X$,  $q$ and $p$, are said to be equivalent if and only if there exist $\kappa>0$ and $\lambda>0$ such that
$$
 \kappa q(x)\le p(x)\le \lambda  q(x)\quad\text{ for all } \,\, x\in X,
$$
(or equivalently,  if and only if $\lambda B_1^q[0]\subset B_1^p[0] \subset \kappa B_1^q[0]$). Obviously, two equivalent asymmetric norms in $X$ generate the same topology. For the proof of the next result it is enough to take into account that for equivalent $q$ and $p$, $q(x)=0$ if and only if $p(x)=0.$

\begin{lemma}\label{l:thetas iguales}
Let $q$ and $p$ be equivalent asymmetric norms in a vector space $X$. Then $\theta_p=\theta_q$.
\end{lemma}


An asymmetric norm $q$ in a vector space $X$ is said to be \textit{right bounded} if and only if there exists $r>0$ such that
$$
rB_1^q[0]\subset B_1^{q^s}[0]+\theta_q.
$$
In this case we also say that $B_1^q[0]$ is right bounded.
In the particular case when $r=1$, we will simply say that $q$ is 1-bounded.
Using (\ref{e:theta balls}) we infer the following

\begin{remark}\label{r: 1 bounded iff balls equals}
Let $(X,q)$ be an asymmetric normed space. The norm
  $q$ is 1-bounded if and only if $B_1^q[0]= B_1^{q^s}[0]+\theta_p$.
\end{remark}

\begin{lemma}\label{l: right bounded iff bounded set}
An asymmetric normed space $(X,q)$ is right bounded if and only if there exists a $q^s$-bounded set $K$ such that
$$B_1^q[0]\subset K+\theta_q$$
\end{lemma}
\begin{proof}
If $(X,q)$ is right bounded, we can find $r>0$ such that $rB_1^q[0]\subset B^{q^s}_1[0]+\theta_q$. Then the set $K:=\frac{1}{r}B^{q^s}_1[0]$ is $q^s$-bounded and
$$B_1^q[0]\subset K+1/r\theta_q=K+\theta_q.$$
For the second implication, let $h>0$ be such that $K\subset h B_1^{q^s}[0]$.
Then
$$B_1^q[0]\subset K+\theta_q\subset h B_1^{q^s}[0]+\theta_q=h B_1^{q^s}[0]+h\theta_q.$$
Thus $1/hB_1^q[0]\subset B^{q^s}+\theta_q$, which proves that $q$ is right bounded.
\end{proof}

\begin{lemma}\label{l:p right bounded q too}
Let $q$ and $p$ be equivalent asymmetric norms in a finite dimensional  vector space $X$. Then $q$ is right bounded if and only if $p$ is right bounded.
\end{lemma}
\begin{proof}
Let us assume that $q$ is right bounded.
Since $X$ is finite dimensional, the norms $q^s$ and $p^s$ are equivalent. Then there exists a positive number $\kappa$ such that $B_1^{q^s}[0]\subset \kappa B_1^{p^s}[0]$.
On the other hand, since $q$ is right bounded there exists  $r'>0$ with the property that $r'B_1^q[0]\subset B^{q^s}_1[0]+\theta_q$.
Finally, since $p$ and $q$ are equivalent, we can find $\lambda>0$ such that $\lambda B_1^p[0]\subset B_1^q[0]$. Now we can use all previous contentions and Lemma~\ref{l:thetas iguales} to conclude that
\begin{align*}
r'\lambda B_1^p[0]&\subset r'B_1^q[0]\subset B_1^{q^s}[0]+\theta_q\\
&\subset  \kappa B_1^{p^s}[0]+\theta_q=\kappa B_1^{p^s}[0]+\theta_p\\
&=\kappa B_1^{p^s}[0]+\kappa\theta_p=\kappa\big(B_1^{p^s}[0]+\theta_p\big).
\end{align*}
Then we infer that $rB_1^p[0]\subset B_1^{p^s}[0]+\theta_p$ where $r=r'\lambda /\kappa$. The other implication can be proven in a similar way.
\end{proof}

\subsection{The canonical $1$-bounded equivalent asymmetric norm}

In  \cite{Conradie 2}  Conradie introduced a new asymmetric norm associated to an asymmetric normed space $(X,p)$ in the following way.  For every $x\in X$ let us define
$$
q_p(x)=\inf \{p^s(x-y)\mid y\in \theta_p\}.
$$
 It was proven in \cite{Conradie 2} that  $q_p:X\to [0,\infty)$ is an asymmetric norm satisfying the following properties:
\begin{itemize}
\item[(p1)] $p(x)\leq q_p(x)$ for all $x\in X$,
\item [(p2)]  $B_1^{q_p}[0]\subset B_1^{p}[0]$,
\item [(p3)]$B_1^{q_p}[0]=B_1^{p^s}[0]+\theta_p$,
\item [(p4)] $p$ is equivalent to $q_p$ if and only if $p$ is right bounded.
\end{itemize}

In the next lemma we prove some other fundamental properties about the norm $q_p$ that will be used later
on.

\begin{lemma} Let $(X,p)$ be an asymmetric normed space. Then the following statements hold
\begin{itemize}
\item [(p5)] $\theta_p=\theta_{q_p}$ (equivalently, $p(x)=0$ if and only if $q_p(x)=0$).
\item [(p6)] $p^s=q_p^s$  (in particular $B_1^{p^s}[0]=B_1^{q_p^s}[0]$).
\item [(p7)] $q_{p}$ is 1-bounded.
\item [(p8)] $B_1^{q_p}[0]=B_1^p[0]$ if and only if $p$ is 1-bounded.
\end{itemize}
\end{lemma}

\begin{proof}

(p5) The inclusion $\theta_{q_p}\subset \theta_p$ follows from property (p1). Now, if $x\in \theta_p$, we can use the definition of $q_p$ to infer that
$$
0\leq q_p(x)\leq p^s(x-x)=p^s(0)=0.
$$
Then $q_p(x)=0$ and therefore $x\in \theta_{q_p}$.

(p6) Let $x\in X$.  By property (p1), we have that $q_p(x)\geq p(x)$ and $q_p(-x)\geq p(-x)$. Thus
\begin{align*}
q_p^s(x)=\max\{q_p(x), q_p(-x)\}\geq\max\{p(x), p(-x)\}=p^s(x).
\end{align*}
On the other hand, the definition of $q_p$ guarantees that $p^s(x)=p^s(x-0)\geq q_p(x)$ and $p^s(-x)=p^s(-x-0)\geq q_p(-x)$.  Since $p^s$ is a norm, $p^s(x)=p^s(-x)$ and therefore
$$ p^s(x)\geq\max\{q_p(x), q_p(-x)\}=q_p^s(x).$$
Then we get $p^s(x)=q_p^s(x)$, as desired.

(p7) Using Remark~\ref{r: 1 bounded iff balls equals} and properties (p3), (p5) and (p6) we get that
$$B_1^{q_p}[0]=B_1^{p^s}[0]+\theta_p=B_1^{q_p^s}[0]+\theta_{q_p}.$$
Then $q_p$ is 1-bounded.

(p8) This is a direct consequence of Remark~\ref{r: 1 bounded iff balls equals} and property (p3).

\end{proof}

\section{Convexity and compactness}

In this section we will show some results relating to compactness properties of subsets of asymmetric normed spaces, mainly for convex sets. The notion of local compactness arises in a natural way in the characterizations, which are given in terms of an equivalent asymmetric norm  satisfying some special properties. 

It is well known that in finite dimensional normed spaces, the convex hull of a compact set is compact. In the following lemma we show the asymmetric version of this result.

\begin{lemma}\label{l:convex hull compact}
Let $A\subset X$ be a  $q$-compact set in the finite dimensional asymmetric normed space $(X,q)$. Then the convex hull $\conv (A)$ is also $q$-compact.
\end{lemma}

\begin{proof}
Let $n$ be the (algebraic) dimension of $X$.  We apply  Carath\'eodory's Theorem (\cite[Theorem 2.2.4]{Webster}) to infer that
$$\conv (A)=\left\{\sum_{i=0}^n\lambda_i a_i : \, \, a_i \in A, \, \lambda_i\geq 0 \; \text{ and }\sum_{i=0}^n\lambda_i=1 \right \}.$$
On the other hand, the topological product  $A^{n+1}$ is also compact since $A$ is compact. Now, let us consider the $n$-dimensional simplex
$$\Delta=\left \{(\lambda_0,\dots, \lambda_n)\in\mathbb R^{n+1}\mid \lambda_i\geq 0\text{ and }\sum_{i=0}^n\lambda_i=1 \right \}.$$
If we equip $\Delta$ with the Euclidean topology, it becomes a compact space.
Now, let $\varphi: A^{n+1}\times \Delta\to \conv (A)$ be the natural map defined by
$$\varphi (a_0,\dots, a_{n},\lambda_0,\dots, \lambda_n)=\sum_{i=0}^n\lambda_i a_i.$$
Clearly $\varphi$ is an onto function.  Now observe that $\varphi$ is continuous since the addition is continuous in $X$ (\cite[Proposition 1.1.40]{cobzas}) and the scalar multiplication by positive numbers is also continuous (Lemma~\ref{l:scalar multiplication continuous}). Using the fact that
 $ A^{n+1}\times \Delta$ is compact we can then conclude that $\conv (A)$ is also compact.
\end{proof}

\begin{lemma}
Let $A$ be a strongly compact set in the finite dimensional asymmetric normed space $(X,q)$. Then $\conv (A)$ is also strongly compact.
\end{lemma}

\begin{proof}
Let $K\subset A$ be a $q^s$-compact set such that
$K\subset A\subset K+\theta_{q}$.
By \cite[Theorem 2.2.6]{Webster}, $\conv(K)$ is $q^s$-compact. Furthermore, we have that
$$\conv (K)\subset \conv (A)\subset \conv (K+\theta_q).$$
Since $\conv (K)+\theta_q$ is a convex set  containing $K+\theta_q$, we conclude that $\conv (K+\theta_q)\subset \conv (K)+\theta_q$ and therefore
$$\conv (K)\subset \conv (A)\subset \conv (K)+\theta_q.$$
This proves that $\conv (A)$ is strongly compact.
\end{proof}

\begin{lemma}\label{l: q^s closure is strongly compact}
Let $A$ be a  strongly compact set in a finite dimensional asymmetric normed space $(X,q)$. Then the $q^s$-closure of $A$ is also strongly compact.
\end{lemma}

\begin{proof}
Let $K\subset X$ be a $q^s$-compact set such that $K\subset A\subset K+\theta_q$. For any $B\subset X$, let us denote by $\overline{B}^s$ the $q^s$-closure of $B$. Then we have that
$$\overline K^s\subset \overline A^s\subset \overline{K+\theta_q}^s.$$
Since $K$ is $q^s$-closed we have that $K=\overline K^s$. On the other hand, since $\theta_q$ is $q^s$-closed (see e.g. \cite[Lemma 3.6]{Jonard Sanchez}) by \cite[Theorem 1.8.10-(ii)]{Webster} the set $K+\theta_q$ is $q^s$-closed and therefore $K+\theta_q=\overline{K+\theta_q}^s$. Then
$ K\subset \overline A^s\subset K+\theta_q$
which proves that $\overline A^s$ is strongly compact, as desired.
\end{proof}

\begin{proposition}\label{p: locally compact norm}
Let $(X,q)$ be a finite dimensional asymmetric normed space.  If the origin has a compact neighbourhood $U$, then there exists an equivalent asymmetric norm $p$ such that each closed ball  $B_r^p[x]$ is  compact. If additionally $U$ is strongly compact, then the balls $B_r^p[x]$ are also strongly compact.
\end{proposition}

\begin{proof}
Let $U$ be a $q$-compact neighbourhood of the origin. Then there exist $\varepsilon>0$ and $\kappa >0$ such that
$$\varepsilon B_{1}^q[0]\subset U\subset \kappa B_{1}^q[0].$$
Let us denote by $A$ the $q^s$-closure of the convex hull of $U$. By Lemma~\ref{l:convex hull compact} and \cite[Lemma 3.2]{Jonard Sanchez}, we have that $A$ is compact.
Since $\kappa B_{1}^q[0]$ is convex and $q^s$-closed, we have that
\begin{equation}\label{f:contencion norma equivalente}
\varepsilon B_{1}^q[0]\subset U\subset A\subset \kappa B_{1}^q[0]ñ
\end{equation}
Now let $p$ be defined as the gauge functional of $A$. Namely,
$$p(x):=\{\inf t>0 \mid x\in tA\}.$$
Since $A$ is absorbent, convex and does not contain any line, $p$ defines an asymmetric norm such that $B_1^{p}[0]=A$ (see \cite{cobzas, Conradie 2}). Furthermore, by (\ref{f:contencion norma equivalente}) $p$ is equivalent to $q$. Since $A$ is $q$-compact, it is also $p$-compact and all the closed balls $B_r^p[x]=rB_1^p[0]+rx$ are also compact.

Now, if $U$ is strongly compact we get by Lemma~\ref{l: q^s closure is strongly compact} that $A=B_1^p[0]$ is strongly compact and therefore all closed balls $B_r^p[x]$ are strongly compact, as desired.
\end{proof}

\section{Strong local compactness}

The aim of this section is to analyze the boundedness properties of strong locally compact asymmetric spaces. As we said in the Introduction, right bounded spaces have not necessarily compact unitary balls. Indeed, let us start with the following example, which is a counterexample to \cite[Proposition 17]{luis}.

\begin{example}\label{e:main 1}
There exists a finite dimensional asymmetric normed space $(X,q)$ such that the unitary closed ball $B_1^q[0]$ is right bounded and non compact.
\end{example}
\begin{proof}

Let $X=\mathbb R^2$. Define the set $A\subset X$ as follows
$$A=\left\{(x,y)\in\mathbb R^2\mid y\leq \frac{1}{x-2}+2,\; x<2\right\}.$$
Let $q:X\to\mathbb [0,\infty)$ be defined as the gauge functional of $A$. Namely
$$q((x,y))=\inf \{t>0\mid( x,y)\in tA\}.$$
Clearly $q$ is an asymmetric norm in $X$ such that $A=B_1^q[(0,0)]$.
Furthermore, it is easy to verify that $\theta_q=\{(x,y)\mid x\leq 0, y\leq 0\}$ and
$B_{1}^q[(0,0)]\subset (2,2)+\theta_q\subset 2B_1^{q^s}[(0,0)]+\theta _q$.
Then  we can conclude that
$$\frac{1}{2}B_{1}^q[(0,0)]\subset B_1^{q^s}[(0,0)]+\theta_q$$
and therefore $q$ is right bounded.
Now let us observe that $B_{1}^q[(0,0)]$ is non compact. Indeed, let us consider the family $\mathcal U=\{B_1^q((2,t))\}_{t<0}$.
Clearly, for every $t<0$,
$$B_1^q((2,t))=\left\{(x,y)\in\mathbb R^2, y<\frac{1}{x-4}+2+t, x<4\right\}$$
and $\mathcal U$ is an open cover for $B_1^q[(0,0)]$.
Furhtermore, if $t>s$, then
$B_1^q((2,s))\subset B_1^q((2,t))$ and therefore $\mathcal U$ is a nested family.
A simple calculation shows  that every point $(x,y)\in B_1^q[(0,0)]$ such that
  $y=1/(x-2)+2$ and $x\geq 3-\sqrt{1-2/t}$   cannot belong to $B_1^q((2,t))$.
Since $\mathcal U$ is nested we conclude that $B_1^q[(0,0)]$ cannot be covered with a finite subfamily of $\mathcal U$ and therefore $B_1^q[(0,0)]$ is non compact.
\end{proof}

\begin{remark}\label{r:locally compact not norm}
The norm $q$ defined in Example~\ref{e:main 1} is equivalent to the asymmetric lattice  norm $p:\mathbb R^2\to \mathbb R$ given by
$$p((x,y))=\max\{x^+, y^+\}$$
where $x^+=\max\{x,0\}$. It is well known that $B_1^p[(0,0)]$ is $p$-compact and therefore $(X,p)$ is locally compact (in fact, it is strongly locally compact). Since $(X,p)$ is equivalent to $(X,q)$ we infer that an asymmetric normed space can be locally compact, even if the unitary closed ball is not compact.
\end{remark}

After Example~\ref{e:main 1} and Remark~\ref{r:locally compact not norm}, it is natural to ask what is the relation between right boundedness and strong local compactness in finite dimensional asymmetric normed spaces. The aim of this section is to answer that question.

\begin{lemma}\label{l: second main implication}
Let $(X,q)$ be a finite dimensional asymmetric normed space. If $q$ is 1-bounded, then $B_1^q[0]$ is strongly compact  and therefore $(X,q)$ is strongly locally compact.

\end{lemma}

\begin{proof}
Observe that $B_1^{q^s}[0]$ is $q^s$-compact (because $X$ is finite dimensional). Now, since $q$ is 1-bounded, we  have that
$$B_1^{q^s}[0]\subset B_1^q[0] =B_1^{q^s}[0]+\theta_q.$$
This directly implies that  $B_1^q[0]$ is strongly compact.

\end{proof}

Many important asymmetric normed spaces are 1-bounded. For example, every asymmetric norm defined by a Banach lattice norm is 1-bounded (\cite[Lemma 1]{todos}). Furthermore, as we have already shown, the norm $q_p$ associated to an asymmetric normed space $(X,p)$ is 1-bounded. This gives us the following

\begin{corollary}
If $(X,p)$ is a finite dimensional asymmetric normed space, then $B_{1}^{q_p}[0]$ is strongly compact.
\end{corollary}

\begin{lemma}\label{l:first main implications}
The asymmetric normed space $(X,p)$ is right bounded if and only if there exists an equivalent norm $q$ such that $q$ is $1$-bounded.
\end{lemma}

\begin{proof}
For the first implication just consider $q=q_p$ and properties (p4) and (p7).
The second implication follows directly from Lemma~\ref{l:p right bounded q too}
\end{proof}

\begin{proposition}\label{p: locally compact is right bounded}
 Let $(X,p)$ be a finite dimensional asymmetric normed space. If the origin has a strongly compact neighborhood, then $p$ is right bounded.
\end{proposition}

\begin{proof}
 By Proposition~\ref{p: locally compact norm}, we can find an equivalent norm $q$ such that $B_1^q[0]$ is strongly compact. Then there exists a $q^s$-compact set $K\subset X$ such that
$$K\subset B_1^q[0]\subset K+\theta_q.$$
Since $X$ is finite dimensional we have that $q^s$ and $p^s$ are equivalent norms, then  $K$ is $p^s$ compact and we can find $a>0$ such that $K\subset aB_1^{p^s}$. On the other hand, since
 $q$ is equivalent to $p$,  there exists $b>0$ with  $bB_1^p[0]\subset B_1^q[0]$.
Using Lemma~\ref{l:thetas iguales}, we infer that
\begin{align*}
bB_1^p[0]&\subset B_1^q[0]\subset K+\theta_q\subset aB_1^{p^s}[0]+\theta_q\\
&=aB_1^{p^s}[0]+a\theta_p=a(B_1^{p^s}[0]+\theta_p).
\end{align*}
Therefore $r B_1^p[0]\subset B_1^{p^s}[0]+\theta_p$ with $r=b/a$, and then $p$ is right bounded, as desired.
\end{proof}

Now we summarize all the previous work in the main result of this section.

\begin{theorem}  \label{elpri}
Let $(X,p)$ be a finite dimensional asymmetric normed space. The following statements are equivalent:
\begin{enumerate}[\rm(1)]
\item $p$ is right bounded.
\item There exists an equivalent asymmetric norm $q$ in $X$ such that $q$ is $1$-bounded.
\item There exists an equivalent asymmetric norm $q$ in $X$ such that $B_1^q[0]$ is strongly compact.

\item $(X,p)$ is strongly locally compact.
\end{enumerate}

\end{theorem}

\begin{proof}
The part $(1)\Longleftrightarrow (2)$ was already proved in Lemma~\ref{l:first main implications}.

Implication $(2)\Longrightarrow (3)$ follows from Lemma~\ref{l: second main implication}.

Finally, implication  $(3)\Longrightarrow (4)$ is  obvious and  $(4)\Longrightarrow (1)$ was proved in Proposition~\ref{p: locally compact is right bounded}.
\end{proof}

We finish this section by giving a  useful characterization of equivalent right bounded asymmetric norms.

\begin{theorem}
Let $X$ be a real finite dimensional linear space. Two right bounded asymmetric norms $p$ and $p'$ are equivalent if and only if $\theta_{p}=\theta_{p'}$.
\end{theorem}

\begin{proof}
The ``only if" implication was already proven in Lemma~\ref{l:thetas iguales}. For the ``if" implication let $q:=q_p$ and $q':=q_{p'}$. Then $q$ and $q'$ are $1$-bounded and therefore $B_{1}^q[0]=B_1^{q^s}[0]+\theta_q$ and $B_{1}^{q'}[0]=B_1^{q'^s}[0]+\theta_q'$.

On the other hand, since $X$ is finite dimensional, the norms $q^s$ and $q'^s$ are equivalent and therefore we can find $\mu>0$ and $\lambda >0$ with the property that
 $$\mu B_{1}^{q^s}[0]\subset B_1^{q'^s}[0]\subset  \lambda B_{1}^{q^s}[0].$$
 Since $\theta_{p}=\theta_{p'}$ we can use property (p5) to infer that
 $\theta_q=\theta_{q'}$. Thus
 \begin{align*}
\mu B_1^q[0]&= \mu (B_{1}^{q^s}[0]+\theta_q)= \mu B_{1}^{q^s}[0]+\mu\theta_q\\
&=\mu B_{1}^{q^s}[0]+\theta_q \subset B_1^{q'^s}[0]+\theta_{q}\\
&\subset  \lambda B_{1}^{q^s}[0]+\theta_q=\lambda B_{1}^{q^s}[0]+\lambda\theta_q\\
& =\lambda ( B_{1}^{q^s}[0]+\theta_q)=\lambda B_{1}^{q}[0].
 \end{align*}
Since $B_{1}^{q'^s}[0]+\theta_q=B_{1}^{q'^s}[0]+\theta_q'=B_1^{q'}[0]$, we conclude that $q$ and $q'$ are equivalent. Now, since $p$ and $p'$ are right bounded, they are equivalent to $q$ and $q'$ (by property (p4)), respectively, and therefore $p$ and $p'$ are also  equivalent.
\end{proof}

\section{Right boundedness and the geometry of the unitary ball}

Let us finish the paper by exploring  the relation among right boundedness, compactness and the geometry of the unitary closed ball in a finite dimensional asymmetric normed space.

In \cite{Jonard Sanchez} it was proven that the convex hull of the extreme points of a compact convex set in an asymmetric normed space defines the topology of the set. In this section we will use the techniques used there to understand the relation among the geometry and the topology of the unitary closed balls in finite dimensional asymmetric normed spaces.

Given a convex set
$K$ in an asymmetric normed space $(X,q)$, let us denote by $E(K)$ the extreme points\footnote{ A point
$x\in A$ is an \textit{extreme} point of $A$ if $x=y=z$ whenever $y,z\in A$ and $x=\lambda y+(1-\lambda)z$ for some $\lambda\in (0,1)$.} of $K+\theta_q$, and  by $S(K)$ the convex hull of $E(K)$.

First, it is interesting to note compactness of the set $E(K)$ is not a necessary condition to guarantee the compactness of a convex set $K$. For example, consider the   the asymmetric lattice norm $(\mathbb{R}^2, p)$ defined in Remark~\ref{r:locally compact not norm}. Then the set
$$K=\{(x,y)\in\mathbb R^2\mid x\leq 0\text{ and }y\leq -x^2\}$$
is a strongly compact convex set with an unbounded set $E(K)$.

On the other hand
 it was proven in \cite{new} that in a finite dimensional asymmetric normed lattice $(X,q)$, all $q$-compact and $q^s$-closed sets are strongly compact.
If the dimension of $X$ is 2, then it is also known (see \cite{Jonard Sanchez 2}) that every $q$-compact and convex set is strongly compact (even if it is not $q^s$-closed). Nevertheless,  not all compact convex sets are strongly compact (the reader can find an example of this in \cite{Jonard Sanchez 2}).
After these results, it is natural to ask if in general, a  compact convex set of a finite dimensional  asymmetric normed space $(X,q) $ is strongly compact if it is $q^s$-closed. Even more, after Theorem~\ref{elpri}  another  natural question arises. On one hand, Theorem \ref{elpri} shows that local compactness of an equivalent norm 
characterizes  right boundedness and strong local compactness. Therefore, we can ask if in general compactness of the unit ball implies strong compactness.
The next example solves these two questions negatively. 

\begin{example} \label{inf}
There is  a finite dimensional asymmetric normed space such that  $B_1^q[0]$ is compact, but not strongly compact.
\end{example}
\begin{proof}
In order to see this, 
consider the cylinder $C=\{(x,y, z)\in\mathbb R^3 \mid y^2+z^2 \leq 1\text{ and }x\leq 1\}$ in $\mathbb R^3$, and let us define
$$R=\{(x,0,0)\in\mathbb R^3\mid x\leq 0\}.$$
Let $p:\mathbb R^3\to [0,\infty)$ be the gauge functional over the set $C$
$$p(u)=\inf\{t\geq 0\mid u\in tC\}.$$
Then $p$ is an asymmetric norm in $\mathbb R^3$, for whom the unitary closed ball is the set $C$ and such that $\theta_p=R$. We will define an equivalent norm $q:\mathbb R^3\to [0,\infty)$ with the required characteristics. 

To do that, let us start by defining, for each $n\in\mathbb N$, the values:
$$y_n= \cos\left(\frac{n\pi}{2(n+1)}\right)\quad \text{ and }\quad z_n= \sin\left(\frac{n\pi}{2(n+1)}\right).$$
And let us consider the sequence $(u_n)_{n\in\mathbb N}$ in $\mathbb R^3$, where $u_n=(-n,y_n, z_n )$. Evidently, in the topology $\tau_p$, the sequence $ (u_n)_{n\in\mathbb N}$ converges to the point $(1,0,0)$, and therefore the set $A_1=\{u_n\mid n\in\mathbb N\}\cup \{(1,0,0)\}$ is $p$-compact.

On the other hand, the set $A_2=\{(1,y,z)\in\mathbb R^3\mid |y|+|z|\leq 1\}$ is clearly $p^s$-compact and therefore it is also $p$-compact. Thus $A_1\cup A_2$ is compact. Let $A=\conv (A_1\cup A_2)$. By Lemma~\ref{l:convex hull compact}, $A$ is compact and therefore $A+\theta_p$ is also compact (Property (O2)). Therefore we can   use \cite[Lemma 3.2]{Jonard Sanchez} in order to conclude that $B:=\overline{A}^s$ is a $p^s$-closed $p$-compact convex set. Furthermore, the set $B$ satisfies 
\begin{enumerate}[\rm(1)]
\item $B=B+\theta_p$
\item  $\frac{1}{\sqrt 2}C\subset B\subset C$
\end{enumerate}

Define now $q:\mathbb R^3\to[0,\infty)$ as the gauge functional over $B$. Namely
$$q(u)=\inf \{t\geq 0\mid u\in tB\}.$$
Clearly $B$ is the closed unitary ball with respect to $q$ and $\theta_q=R=\theta_p$.
Property $(2)$ implies that  $q$ is equivalent to $p$. In particular, since  $p$ is obviously right bounded,  we also get  that $q$ is right bounded (by Lemma~\ref{l:p right bounded q too}). By the same reason, since $B$ is $p$-compact it is also $q$-compact, as desired.

It just remains to prove that $B$ is not strongly compact. For that purpose, suppose the opposite. Then we can find a $q^s$-compact set $K$ such that $K\subset B\subset K+\theta_q$. Since $B$ is convex, we can use Lemma~\ref{l:convex hull compact} in order to assume without any loss of generality that the set $K$ is convex. 
We can also observe that 
$$K\subset B\subset K+\theta_q\subset B+\theta_q=B,$$
and therefore $B=K+\theta_q$. Now, by \cite[Theorem 4.1]{Jonard Sanchez} the set of extreme points $E(B)$ is contained in $K$. Since every point in the set $A_1$ is an extreme point of $B$, we infer that $A_1\subset K$. In particular, we get that $A_1$ is $q^s$-bounded, which is a contradiction.

\end{proof}

In particular this gives the existence of $q^s$-closed convex sets in finite dimensional asymmetric normed spaces that are compact but not strongly compact (c.f. \cite{new}).  Example \ref{inf} also shows   that there is a compact unitary closed ball of an asymmetric normed space of dimension $3$ such that the set of its extreme points 
 $E(B_1^q[0])$ is not $q^s$-bounded.

In what follows we show that some positive results can also be found on the relation among extreme points and right boundedness.

Other results on the behavior of asymmetric norms in finite dimensional spaces can be proven using the tools developed above.
\begin{lemma}\label{l:geometry of balls}
 For any finite dimensional asymmetric normed space
$(X,q)$, we always have $B_1^q[0]=S(B_1^q[0])+\theta_q$.
\end{lemma}

\begin{proof}
Using a well-known result of Klee \cite{Klee}, we know that $B_1^q[0]$ is the convex hull of its extreme points and its extreme rays\footnote{An open half line $R=\{a+tb\mid a,b\in X, ~t> 0\}$ is called an \textit{extreme ray} of $A$ if $y,z\in R$ whenever $\lambda y+(1-\lambda)z\in R$, where  $y, z\in A$ and $\lambda\in (0,1)$. In this case, if $R$ is an extreme ray and the extreme $a$ lies in $A$, then $a$ is an extreme point of $A$.}.
This implies that $S(B_1^q[0])\subset B_1^q[0]$ and therefore
$$S(B_1^q[0])+\theta_q\subset B_1^q[0]+\theta_q=B_1^q[0].$$

In order to prove the other contention, let $x\in B_1^q[0]$. By \cite[Lemma 3.6]{Jonard Sanchez} every extreme ray $B_1^q[0]$ is parallel to some ray contained in $\theta_q$.  Then we can find
$a_1,\dots, a_n,a_{n+1},\dots, a_p\in E(B_1^q[0])$, $b_1,\dots, b_n\in\theta_q$, and scalars $\lambda_1,\dots, \lambda_p,t_1,\dots, t_n\geq 0$ such that
$$x=\sum_{i=1}^{n}\lambda_i (a_i+t_ib_i)+\sum_{i=n+1}^p\lambda_i a_i \quad \text{ and }\quad \sum_{i=1}^p\lambda_i=1.$$
Since $\theta_q$ is a convex cone, the point $\sum_{i=1}^{n}\lambda_i t_ib_i$ lies in $\theta_q$. On the other hand, we know  that $\sum_{i=1}^p\lambda_i a_i\in S(B_1^q[0])$. Thus

\begin{align*}
x&=\sum_{i=1}^{n}\lambda_i (a_i+t_ib_i)+\sum_{i=n+1}^p\lambda_i a_i\\
&=\sum_{i=1}^p\lambda_i a_i+\sum_{i=1}^{n}\lambda_i t_ib_i\\
&\in S(B_1^q[0])+\theta_q.
\end{align*}

\end{proof}

\begin{proposition}\label{l: ps bounded implies right bounded}
Let $(X,q)$ be an asymmetric normed space such that $E(B_1^q[0])$ is $q^s$-bounded. Then $q$ is right bounded.
\end{proposition}

\begin{proof}
If $E(B_1^q[0])$ is $q^s$-bounded, so is $S(B_1^q[0])$. Thus, by Lemmas~\ref{l: right bounded iff bounded set} and \ref{l:geometry of balls} we conclude that $q$ is right bounded.
\end{proof}

The converse of Proposition~\ref{l: ps bounded implies right bounded} is false. A counterexample of this situation was already shown in Example~\ref{e:main 1}, where the norm is right bounded, but the set of extreme points of $B_1^q[0]$ is not $q^s$-bounded.

However, if the norm is $1$-bounded, the situation is different as we can see in the following

\begin{proposition}
Let $(X,q)$ be a finite dimensional $1$-bounded asymmetric normed space. Then $E(B^{q}_1[0])\subset B_1^{q^s}[0]$ and therefore $E(B^{q}_1[0])$ is $q^s$-bounded.
\end{proposition}

\begin{proof}
Since $q$ is 1-bounded, the set $B^{q}_1[0]=B_1^{q^s}[0]+\theta_0$. Since $B_1^{q^s}[0]$ is $q^s$-compact (and $q$-compact), we can use  \cite[Theorem 4.1]{Jonard Sanchez} to infer that  $E(B^{q}_1[0])\subset B_1^{q^s}[0]$ as desired.
\end{proof}

Let us finish the paper with the following open question, that is suggested by 
the strong relation between right boundedness and local compactness.

\begin{question}
Let $(X,q)$ be a finite dimensional asymmetric normed space. If $X$ is locally compact, is $q$ right bounded? In particular, if $B_1^q[0]$ is compact, is $q$ right bounded?
\end{question}

\vspace{2cm}

\noindent[Natalia Jonard-P\'erez]  Departamento de Matem\'aticas, Facultad de Ciencias,  Universidad Nacional Aut\'onoma de M\'exico. Circuito Exterior S/N, Cd. Universitaria, Colonia Copilco el Bajo, Delegaci\'on Coyoac\'an,  04510,  M\'exico D.F., M\'exico, e-mail: nat@ciencias.unam.mx

\vspace{1cm}

\medskip

\noindent[Enrique A. S\'anchez P\'erez] Instituto Universitario de Matem\'{a}tica Pura y Aplicada, Universitat Polit\`ecnica de Val\`encia, Camino de Vera s/n, 46022 Valencia, Spain, e-mail: easancpe@mat.upv.es

\end{document}